\documentclass[11pt]{amsart}
\usepackage{amssymb}
\usepackage{amsmath}
\usepackage[active]{srcltx}
\usepackage{t1enc}
\usepackage[latin2]{inputenc}
\usepackage{verbatim}
\usepackage{amsmath,amsfonts,amssymb,amsthm}
\usepackage[mathcal]{eucal}
\usepackage{enumerate}
\usepackage[centertags]{amsmath}
\usepackage{graphics}

\newtheorem{theorem}{Theorem}

\newtheorem{lemma}{Lemma}
\newtheorem{remark}{Remark}

\newtheorem{corollary}{Corollary}

\begin{document}
\author{L.-E. Persson, G. Tephnadze, P. Wall}
\title[logarithmic means \dots]{On the Nörlund logarithmic means with
respect to Vilenkin system in the martingale Hardy space $H_{1}$}
\address{Department of Engineering Sciences and Mathematics, Lule\aa\ %
University of Technology, SE-971 87 Lule\aa , Sweden and UiT The Arctic
University of Norway, P.O. Box 385, N-8505, Narvik, Norway.}
\email{Lars-Erik.Persson@ltu.se}
\address{G. Tephnadze, The University of Georgia, school of Informatics, Engineering and
Mathematics, IV, 77a Merab Kostava St,
Tbilisi, 0128, Georgia, \& Department of Engineering Sciences and
Mathematics, Lule\aa University of Technology, SE-971 87 Lule\aa , Sweden.}
\email{giorgitephnadze@gmail.com}
\address{P. Wall, Department of Engineering Sciences and Mathematics, Lule%
\aa University of Technology, SE-971 87 Lule\aa , Sweden.}
\email{Peter.Wall@ltu.se}
\date{}
\thanks{The research was supported by Shota Rustaveli National Science
Foundation grant YS15-2.1.1-47, by a Swedish Institute scholarship no.
24155/2016}
\maketitle

\begin{abstract}
In this paper we prove and discuss a new divergence result of Nörlund
logarithmic means with respect to Vilenkin system in Hardy space $H_1. $
\end{abstract}

\date{}

\textbf{2000 Mathematics Subject Classification.} 42C10.

\textbf{Key words and phrases:} Vilenkin system, Nörlund logarithmic means,
partial sums, modulus of continuity, Hardy space.

\section{Introduction}

It is well-known that (see e.g. \cite{AVD} and \cite{gol}) Vilenkin systems
do not form bases in the space $L_{1}.$ Moreover, there exists a function in
the dyadic Hardy space $H_{1},$ such that the partial sums of $f$ are not
bounded in $L_{1}$-norm.

In \cite{tep7} (see also \cite{tepthesis}) it was proved that the following
is true:

\textbf{Theorem T1: }The\textbf{\ }maximal operator $\widetilde{S}^{\ast }$
defined by 
\begin{equation*}
\widetilde{S}^{\ast}:=\sup_{n\in \mathbb{N}}\frac{\left\vert
S_{n}\right\vert }{\log \left( n+1\right)}
\end{equation*}
is bounded from the Hardy space $H_{1}$ to the space $L_{1}.$ Here $S_n $
denotes the $n $-th partial sum with respect to the Vilenkin system.
Moreover, it was proved that the rate of the factor $log(n+1)$ is in a sense
sharp.

M\'oricz and Siddiqi \cite{Mor} investigate the approximation properties of
some special N\"orlund means of Walsh-Fourier series of $L_{p}$ functions in
norm. Fridli, Manchanda and Siddiqi \cite{FMS} improved and extended the
results of M\'oricz and Siddiqi \cite{Mor} to Martingale Hardy spaces. However, the
case when $\left\{q_k=1/k:k\in\mathbb{N}_+\right\}$ was excluded, since
the methods are not applicable to N\"orlund logarithmic means. In \cite{Ga2}
Gát and Goginava proved some convergence and divergence properties of
Walsh-Fourier series of the N\"orlund logarithmic means of functions in the
Lebesgue space $L_{1}.$ In particular, they proved that there exists an
function in the space $L_1, $ such that 
\begin{equation*}
\sup_{n\in \mathbb{N}}\left\Vert L_{n}f\right\Vert _{1}=\infty .
\end{equation*}

Analogical result for some unbounded Vilenkin systems was proved in \cite{bg1}.

In \cite{PTW} (see also \cite{tepthesis}) it was proved that there exists a martingale $f\in H_{p}, \ \ (0<p<1)$ such that 
\begin{equation*}
\sup_{n\in \mathbb{N}}\left\| L_{n}f\right\| _{p}=\infty,
\end{equation*}
Here $L_{n} $ is $n $-th N\"orlund logarithmic means with respect to
Vilenkin system.

In this paper we prove an analogical result for the bounded Vilenkin systems in the
case when $p=1.$ Moreover, we discuss boundedness of weighted maximal
operators on the Hardy space $H_1.$

\section{Preliminaries}

Let $\mathbb{N}_{+}$ denote the set of the positive integers, $\mathbb{N}:= 
\mathbb{N}_{+}\cup \{0\}.$

Let $m:=(m_{0},m_{1},\dots)$ denote a sequence of the positive numbers not
less than 2.

Denote by 
\begin{equation*}
Z_{m_{k}}:=\{0,1,\dots ,m_{k}-1\}
\end{equation*}%
the additive group of integers modulo $m_{k},$ $k\in \mathbb{N}.$

Define the group $G_{m}$ as the complete direct product of the group $%
Z_{m_{k}}$ with the product of the discrete topologies of $Z_{m_{k}}`$s.

The direct product $\mu$ of the measures\ 
\begin{equation*}
\mu _{k}\left( \{j\}\right):=1/m_{k},(j\in Z_{m_{k}})
\end{equation*}
is the Haar measure on $G_{m},$ with $\mu \left( G_{m}\right) =1.$

If $\sup\limits_{n}m_{n}<\infty $, then we call $G_{m}$ a bounded Vilenkin group. If the generating sequence $m$ is not bounded, then $G_{m}$ is said to be an unbounded Vilenkin group.
In this paper we discuss bounded Vilenkin groups only.

The elements of $G_{m}$ are represented by sequences 
\begin{equation*}
x:=(x_{0},x_{1},\dots,x_{k},\dots),\ \left(x_{k}\in Z_{m_{k}}\right).
\end{equation*}

It is easy to give a base for the neighbourhood of $G_{m}:$ 
\begin{equation*}
I_{0}\left( x\right):=G_{m},
\end{equation*}
\begin{equation*}
I_{n}(x):=\{y\in G_{m}\mid y_{0}=x_{0},\dots,y_{n-1}=x_{n-1}\}, (x\in G_{m},
n\in \mathbb{N}).
\end{equation*}
Denote $I_{n}:=I_{n}\left( 0\right),$ for $n\in \mathbb{N}$ and $\overset{-}{%
I_{n}}:=G_{m}\backslash I_{n}$.

The norm (or quasi-norm) of the spaces $L_{p}(G_{m})$ is defined by 
\begin{equation*}
\left\Vert f\right\Vert _{p}:=\left( \int_{G_{m}}\left\vert f\right\vert
^{p}d\mu \right) ^{1/p}\ \ \ \left( 0<p<\infty \right).
\end{equation*}

If we define the so-called generalized powers system based on $m$ in the
following way: 
\begin{equation*}
M_{0}:=1,\ M_{k+1}:=m_{k}M_{k}\ \qquad (k\in \mathbb{N}),
\end{equation*}%
then every $n\in \mathbb{N}$ can be uniquely expressed as $n=\overset{\infty 
}{\underset{k=0}{\sum }}n_{k}M_{k},$ where $n_{k}\in Z_{m_{k}}\ (k\in 
\mathbb{N})$ and only a finite number of $n_{k}`$s differ from zero. Let $%
\left\vert n\right\vert :=\max \{k\in \mathbb{N}:\ n_{k}\neq 0\}.$

Next, we introduce on $G_{m}$ an orthonormal system, which is called the
Vilenkin system. First define the complex valued functions $r_{k}\left(
x\right):G_{m}\rightarrow \mathbb{C},$ the generalized Rademacher functions,
by 
\begin{equation*}
r_{k}\left( x\right):=\exp \left( 2\pi \imath x_{k}/m_{k}\right),\ \left(
\imath ^{2}=-1,\ x\in G_{m}, k\in \mathbb{N} \right).
\end{equation*}

Now, define the Vilenkin systems $\psi:=(\psi _{n}:n\in \mathbb{N})$ on $%
G_{m} $ as: 
\begin{equation*}
\psi _{n}(x):=\prod_{k=0}^{\infty}r_{k}^{n_{k}}\left( x\right),\ \left( n\in 
\mathbb{N}\right).
\end{equation*}

Specifically, we call this system the Walsh-Paley one if $m\equiv 2$.

The Vilenkin systems are orthonormal and complete in $L_{2}\left(
G_{m}\right)$ (see e.g. $\,$\cite{AVD,Vi}).

If $f\in L_{1}\left( G_{m}\right) $ we can establish Fourier coefficients,
partial sums, Dirichlet kernels, with respect to Vilenkin systems in the
usual manner: 
\begin{equation*}
\widehat{f}\left( n\right) :=\int_{G_{m}}f\overline{\psi }_{n}d\mu ,\
\left(k\in \mathbb{N}\right),
\end{equation*}%
\begin{equation*}
S_{n}f:=\sum_{k=0}^{n-1}\widehat{f}\left( k\right) \psi _{k},\text{ \ } \ \
D_{n}:=\sum_{k=0}^{n-1}\psi _{k},\ \ \ \left( k\in \mathbb{N}\right) .
\end{equation*}%
Let $\left\{ q_{k}:k\geq 0\right\} $ be a sequence of nonnegative numbers.
The $n $-th Nörlund mean for the Fourier series of $f$ is defined by 
\begin{equation*}
t_nf=\frac{1}{{l_{n}}}\sum_{k=0}^{n-1}q_{n-k}S_{k}f.
\end{equation*}%
If $q_{k}={1}/{k}$, then we get the Nörlund logarithmic means: 
\begin{equation*}
L_{n}f:=\frac{1}{l_{n}}\sum_{k=0}^{n}\frac{S_{k}f}{n-k}, \ \
l_{n}:=\sum_{k=1}^{n}\frac{1}{k},
\end{equation*}
The kernel of the Nörlund logarithmic means is defined by 
\begin{equation*}
F_n:=\frac{1}{l_{n}}\sum_{k=0}^{n}\frac{D_k}{n-k}.
\end{equation*}
In the special case when $\{q_{k}=1:k\in N\},$ we get the Fej\'er means 
\begin{equation*}
\sigma _{n}f:=\frac{1}{n}\sum_{k=1}^{n}S_{k}f\,.
\end{equation*}
We define $n $-th Fej\'er kernel 
\begin{equation*}
K_{n}f:=\frac{1}{n}\sum_{k=1}^{n}D_{k}.
\end{equation*}
Let $f\in L_{1}(G_{m}).$ Then the maximal function is given by 
\begin{equation*}
f^{\ast}(x)=\sup_{n\in\mathbb{N}}\frac{1}{\left\vert I_{n}\left( x\right)
\right\vert }\left\vert \int_{I_{n}\left( x\right) }f\left( u\right) \mu
\left( u\right) \right\vert .
\end{equation*}
The Hardy martingale spaces $H_{1}$ consist of all martingales for which 
\begin{equation*}
\left\Vert f\right\Vert _{H_{1}}:=\left\Vert f^{\ast }\right\Vert
_{1}<\infty .
\end{equation*}
It it well-known (see e.g. \cite{sws} and \cite{We1}) that if $f\in L_{1},$
then 
\begin{equation}  \label{normsup}
\left\Vert f\right\Vert _{H_{1}}\sim \left\Vert \sup_{n\in \mathbb{N}%
}\left\vert S_{M_{n}}f\right\vert \right\Vert_{1}.
\end{equation}
A bounded measurable function $a$ is a  1-atom if  either $ a=1 $ or
\begin{equation*}
\int_{I}ad\mu =0,\ \left\Vert a\right\Vert _{\infty }\leq \mu \left(
I\right) ,\text{ \ supp}\left( a\right) \subset I.
\end{equation*}

\section{Auxiliary propositions}

The Hardy martingale space $H_{1}\left( G_{m}\right) $ has an atomic
characterization (see \cite{We1}, \cite{We3}):

\begin{lemma}
\label{lemma2.1} A function $f\in H_{1}$ if and only if there exist a
sequence $\left( a_{k},k\in \mathbb{N}\right) $ of 1-atoms and a sequence $%
\left( \mu _{k},k\in \mathbb{N}\right) $ of real numbers such that

\begin{equation}
\qquad \sum_{k=0}^{\infty }\mu _{k}a_{k}=f,\text{ \ a.e.,}  \label{condmart}
\end{equation}
where
\begin{equation*}
\qquad \sum_{k=0}^{\infty }\left\vert \mu _{k}\right\vert <\infty .
\end{equation*}%
Moreover, 
\begin{equation*}
\left\Vert f\right\Vert _{H_{1}}\backsim \inf \left( \sum_{k=0}^{\infty
}\left\vert \mu _{k}\right\vert \right),
\end{equation*}%
where the infimum is taken over all decomposition of $f$ of the form (\ref%
{condmart}).
\end{lemma}

In Blahota, Gát \cite{bg1} the following lemma (see Lemma 5) was proved for unbounded Vilenkin systems:
\begin{lemma} \label{bt}
Let $q_A=M_{2A}+M_{2A-2}+...+M_{0}. $ If $  \log m_n = O(n^{\beta}),$ for some $ 0 < \delta < 1/2, $ then
\begin{equation*}
\left\Vert F_{q_A} \right\Vert_1\geq c(\log q_A)^{\beta},
\end{equation*}
for some $ 0<\beta <1-\delta .$
\end{lemma}

For the proof of main result we also need the following new Lemma of
independent interest:

\begin{lemma} \label{ptw1}
Let $ G_m $ be a bounded Vilenkin system and $q_A=M_{2A}+M_{2A-2}+...+M_{0}. $ Then 
\begin{equation*}
\left\Vert F_{q_A} \right\Vert_1\geq c\log q_A.
\end{equation*}
\end{lemma}

\begin{proof} During the proof we use some method of Blahota, Gát \cite{bg1} (see Lemma 5) and Gát, Goginava \cite{Ga2} (see Lemma 2).

Set 
\begin{equation*}
\theta_{n}:=l_{n}F_{n}=\overset{n-1}{\underset{k=1}{\sum }}\frac{D_{n-k}}{k}.
\end{equation*}

Then we have that 
\begin{eqnarray*}
\theta_{q_A}(x)&=&\overset{M_{2A-2}+...+M_{0}-1}{\underset{k=1}{\sum }}\frac{%
1}{k}D_{M_{2A}+M_{2A-2}+...+M_{0}-k}(x) \\
&+&\overset{M_{2A}+...+M_{0}-1}{\underset{k=M_{2A-2}+...+M_{0}}{\sum }}\frac{%
1}{k}D_{M_{2A}+M_{2A-2}+...+M_{0}-k}(x) \\
&:=& I+II.
\end{eqnarray*}

We first discuss $I.$ Since $k<M_{2A-2}+...+M_0,$ then 
\begin{eqnarray*}
D_{M_{2A}+...+M_0-k}(x)=D_{M_{2A}}(x)+r_{2A}D_{M_{2A-2}+...+M_0-k}(x).
\end{eqnarray*}

This gives that 
\begin{eqnarray}  \label{nordn1}
I=l_{q_{A-1}}D_{M_{2A}}(x)+r_{2A}G_{M_{2A-2}+...+M_{0}}(x).
\end{eqnarray}

Moreover, by using Abel transformation, we get that 
\begin{eqnarray*}
II&=&\overset{M_{2A}+...+M_{0}-1}{\underset{k=M_{2A-2}+...+M_{0}}{\sum }}%
\frac{1}{k}D_{M_{2A}+M_{2A-2}+...+M_{0}-k} \\
&=& \frac{K_{1}}{M_{2A}+...+M_{0}-1}-\frac{\left(M_{2A}+...+M_{0}-2%
\right)K_{M_{2A}+...+M_{0}-2}}{M_{2A}+...+M_{0}-1} \\
&+&\overset{M_{2A}+...+M_{0}-2}{\underset{k=M_{2A-2}+...+M_{0}}{\sum }}\frac{%
M_{2A}+M_{2A-2}+...+M_{0}-k}{k(k+1)}K_{M_{2A}+M_{2A-2}+...+M_{0}-k}.
\end{eqnarray*}
Since (for details see e.g. \cite{AVD}) $\left\Vert K_{n}\right\Vert
_{1}\leq 2,$ for all $n\in \mathbb{N},$ we obtain that 
\begin{eqnarray}
&&\left\Vert II\right\Vert _{1}\leq \frac{2}{M_{2A}+...+M_{0}-1}+\frac{%
2\left( M_{2A}+...+M_{0}-2\right) }{M_{2A}+...+M_{0}-1}  \label{nordn2} \\
&+&2\overset{M_{2A}+...+M_{0}-2}{\underset{k=M_{2A-2}+...+M_{0}}{\sum }}%
\frac{M_{2A}+M_{2A-2}+...+M_{0}-k}{k(k+1)}<c<\infty .  \notag
\end{eqnarray}
Hence, by (\ref{nordn1}) and (\ref{nordn2}), we get that 
\begin{equation*}
\left\Vert \theta_{q_A} \right\Vert_1\geq \left\Vert
l_{q_{A-1}}D_{M_{2A}}+r_{2A}\theta_{q_{A-1}} \right\Vert_1-c.
\end{equation*}

We now discuss the right-hand side of this inequality, more exactly we give
a lower bound for it. First, we consider the integral on the set $G_m\setminus
I_{2A}$: 
\begin{eqnarray*}
&&\int_{G_m\setminus I_{2A}} {\left\vert
l_{q_{A-1}}D_{M_{2A}}(x)+r_{2A}\theta_{q_{A-1}}(x) \right\vert}d\mu(x) \\
&=&\int_{G_m\setminus I_{2A}} {\left\vert \theta_{q_{A-1}}(x) \right\vert}%
d\mu(x)=\left\Vert \theta_{q_{A-1}} \right\Vert_1-\int_{I_{2A}} {\left\vert
\theta_{q_{A-1}}(x) \right\vert}d\mu(x) \\
&=&\left\Vert \theta_{q_{A-1}} \right\Vert_1-\frac{ \theta_{q_{A-1}}(0) }{%
M_{2A}}.
\end{eqnarray*}
Next, we note that on the set $I_{2A}$ we have that 
\begin{eqnarray*}
\int_{I_{2A}} {\left\vert l_{q_{A-1}}D_{M_{2A}}(x)+r_{2A}\theta_{q_{A-1}}(x)
\right\vert}d\mu(x) =l_{q_{A-1}}-\frac{ \theta_{q_{A-1}}(0) }{M_{2A}}.
\end{eqnarray*}
It follows that 
\begin{equation}  \label{norgm1}
\left\Vert \theta_{q_A} \right\Vert_1\geq \left\Vert \theta_{q_{A-1}}
\right\Vert_1+l_{q_{A-1}}-\frac{ 2\theta_{q_{A-1}}(0)}{M_{2A}}-c.
\end{equation}
Moreover, according to simple estimation, 
\begin{equation*}
\theta_{n}(0)=\overset{n-1}{\underset{k=1}{\sum }}\frac{n-k}{k}=n \overset{%
n-1}{\underset{k=1}{\sum }}\frac{1}{k}-n+1=nl_n +O(n).
\end{equation*}
Therefore, since 
\begin{equation*}
q_{A-1}\leq M_{2A-2}(1+\frac{1}{4}+\frac{1}{16}+...)= \frac{4}{3}%
M_{2A-2}\leq \frac{1}{3}M_{2A}
\end{equation*}%
we obtain that    
\begin{equation}
\frac{2}{l_{q_{A}}M_{2A}}\theta _{q_{A-1}}(0)=\frac{2q_{A-1}}{M_{2A}}\frac{%
l_{q_{A-1}}}{l_{q_{A}}}+o(1)\leq \frac{2}{3}\frac{l_{q_{A-1}}}{l_{q_{A}}}%
+o(1).  \label{norgm2}
\end{equation}
Finally, by using (\ref{norgm1}) and (\ref{norgm2}), we conclude that 
\begin{eqnarray*}
\left\Vert F_{q_{A}}\right\Vert _{1} &=&\frac{1}{l_{q_{A}}}\left\Vert \theta
_{q_{A}}\right\Vert _{1} \\
&\geq &\frac{1}{l_{q_{A}}}\left\Vert \theta _{q_{A-1}}\right\Vert _{1}+\frac{%
l_{q_{A-1}}}{l_{q_{A}}}-\frac{2}{3}\frac{l_{q_{A-1}}}{l_{q_{A}}}-o(1) \\
&\geq &\frac{1}{l_{q_{A}}}\left\Vert \theta _{q_{A-1}}\right\Vert _{1}+\frac{%
1}{3}\frac{l_{q_{A-1}}}{l_{q_{A}}}-o(1) \\
&\geq &\frac{1}{l_{q_{A}}}\left\Vert \theta _{q_{A-1}}\right\Vert _{1}+\frac{%
1}{6}-o(1) \\
&\geq &\frac{l_{q_{A-1}}}{l_{q_{A}}}\left\Vert F_{q_{A-1}}\right\Vert _{1}+%
\frac{1}{6}-o(1) \\
&\geq &\frac{l_{q_{A-1}}}{l_{q_{A}}}\left\Vert F_{q_{A-1}}\right\Vert _{1}+%
\frac{1}{8}
\end{eqnarray*}

By iterating this estimate we get 
\begin{equation*}
\left\Vert F_{q_{A}}\right\Vert _{1}\geq \frac{1}{8}\overset{A-1}{\underset{%
k=1}{\sum }}\frac{l_{q_{k}}}{l_{q_{A}}}\geq \frac{c}{8A}\overset{A-1}{%
\underset{k=1}{\sum }}k\geq cA\geq c\log q_{A}.
\end{equation*}

The proof is complete.
\end{proof}

\section{Formulation of the Main Results}

Our main results read:

\begin{theorem}
Let $ G_m $ be a bounded Vilenkin system. Then there exist a martingale $f\in H_{1}$ such that 
\begin{equation} \label{main}
\sup\limits_{n\in \mathbb{N}}\left\Vert L_{n}f\right\Vert_{1}=+\infty .
\end{equation}
\end{theorem}

\begin{remark}
In one point of view Theorem 1 of Blahota and Gát \cite{bg1} is better then our resut (see Theorem 1) because in their result boundedness of the group $G_m$ is not necessary and in the other point of view theorem of Blahota and Gát is a slightly weaker since they construct function in the Lebesgue space  $L_1 $.

\end{remark}

The next result can be found in \cite{tepthesis} and \cite{PTW}, respectively:

\begin{corollary}
Let $ G_m $ be a bounded Vilenkin system. Then there exists a martingale $f\in H_{1}$ such that 
\begin{equation*}
\left\Vert L^{\ast }f\right\Vert _{1}=+\infty.
\end{equation*}
\end{corollary}

\begin{corollary}
Let $0<p<1$ and $ G_m $ be a bounded Vilenkin system. Then there exists a martingale $f\in H_{p}$ such that 
\begin{equation*}
\left\Vert L^{\ast }f\right\Vert _{p}=+\infty .
\end{equation*}
\end{corollary}

\begin{theorem}
\label{theorem1}a) Let $ G_m $ be a bounded Vilenkin system. Then the maximal operator 
\begin{equation*}  \label{wemax}
\overset{\sim }{L}^{\ast }f:=\sup_{n\in \mathbb{N}}\frac{\left\vert
L_{n}f\right\vert }{\log\left(n+1\right)}
\end{equation*}
is bounded from the Hardy space $H_{1}\left( G_{m}\right) $ to the space $%
L_{1}\left( G_{m}\right) .$

b) Let $\varphi :\mathbb{N}_{+}\rightarrow \lbrack 1,\infty )$ be a
nondecreasing function satisfying the condition%
\begin{equation}
\overline{\lim_{n\rightarrow \infty }}\frac{\log \left( n+1\right) }{\varphi
\left( n\right) }=+\infty .  \label{5}
\end{equation}%
Then there exists a martingale $f\in H_{1}\left( G_{m}\right),$ such that the maximal operator 
\begin{equation*}
\sup_{n\in \mathbb{N}}\frac{\left\vert L_{n}f\right\vert }{\varphi \left(
n\right) }
\end{equation*}%
is not bounded from the Hardy space $H_{1}\left(G_{m}\right) $ to the
Lebesgue space $L_{1}\left( G_{m}\right) .$
\end{theorem}

\begin{remark}
\label{remark1} Note that b) shows that the statement in part a) of Theorem
2 is in the sense sharp with respect to the logarithmic factor in (\ref%
{wemax}).
\end{remark}

\section{Proofs of the Main Results}

\begin{proof}[Proof of Theorem 1.]
Let $\left\{ \alpha _{k}:k\in \mathbb{N}\right\} $ be an increasing sequence
of positive integers such that 
\begin{equation}
\sum_{\eta =0}^{k-1}\frac{M_{2\alpha_{\eta}}}{ \alpha^{1/2} _{\eta }}<\frac{%
M_{2\alpha _{k}}}{\alpha^{1/2}_{k}},  \label{4}
\end{equation}

\begin{equation}
\frac{M_{2\alpha_{k-1}}}{\alpha^{1/2}_{k-1}}<\alpha _{k}.  \label{55}
\end{equation}

Since  $$\alpha _{k}\rightarrow \infty, \ \  \sum_{\eta =0}^{k-1}\frac{M_{2\alpha_{\eta}}}{ \alpha^{1/2} _{\eta }}\rightarrow \infty, \ \ \frac{%
M_{2\alpha _{k}}}{\alpha^{1/2}_{k}}\rightarrow \infty,\ \ \text{as} \ \ k\rightarrow \infty$$  and all of three sequences are strightly increasing, we can construct lacunar increasing sequence $\left\{ \alpha _{k}:k\in\mathbb{N}%
\right\} $ of natural number $ \mathbb{N} $, for which 
\begin{equation*}
\sum_{\eta =0}^{k-1}\frac{M_{2\alpha_{\eta}}}{\alpha^{1/2}_{\eta }}<\frac{M_{2\alpha_{k}}}{\alpha^{1/2}_{k}}<\alpha_{k+1}.
\end{equation*}

We also note that under condition \eqref{55} we get that 
\begin{equation*}
M_{2\alpha_{k-1}}<\alpha_{k}\alpha^{1/2}_{k-1}\leq \alpha_{k}^{3/2}
\end{equation*}
and
\begin{equation}
\sum_{k=0}^{\infty }\alpha _{k}^{-1/2}<\alpha_{0}+\sum_{k=0}^{\infty }\frac{1}{M_{2\alpha_{k-1}}^{1/3}}<c<\infty.  \label{3}
\end{equation}

Similar consructions of martingales can be found in \cite{tep7}-\cite{tepthesis} (see also \cite{NT1}).

Let \qquad 
\begin{equation*}
f\left(x\right)=\sum_{k=0}^{\infty}\lambda_{k}a_{k},
\end{equation*}
where 
\begin{equation*}
\lambda _{k}=\frac{m_{2\alpha _{k}}}{\alpha _{k}}, \ \ a_{k}=\frac{1}{%
m_{2\alpha _{k}}}\left( D_{M_{2\alpha _{k}+1}}-D_{M_{_{2\alpha _{k}}}}
\right).
\end{equation*}
From (\ref{3}) and Lemma 1 we can conclude that $f\in
H_{1}\left(G_{m}\right).$
It is easy to show that 
\begin{equation}
\widehat{f}(j)=\left\{ 
\begin{array}{l}
\frac{1}{\alpha^{1/2} _{k}},\,\,\text{ if \thinspace \thinspace }j\in
\left\{ M_{2\alpha _{k}},...,\text{ ~}M_{2\alpha _{k}+1}-1\right\} ,\text{ }%
k=0,1,2..., \\ 
0,\text{ \thinspace \thinspace \thinspace if \thinspace \thinspace
\thinspace }j\notin \bigcup\limits_{k=1}^{\infty }\left\{ M_{2\alpha
_{k}},...,\text{ ~}M_{2\alpha _{k}+1}-1\right\} \text{ .}%
\end{array}
\right.  \label{8}
\end{equation}
We can write that 
\begin{eqnarray}
L_{q_{\alpha _{k}}}f &=&\frac{1}{l_{q_{\alpha _{k}}}}\underset{j=1}{\overset{%
q_{\alpha _{k}}}{\sum }}\frac{S_{j}f}{q_{\alpha _{k}}-j}  \label{16} \\
&=&\frac{1}{l_{q_{\alpha _{k}}}}\underset{j=1}{\overset{M_{2\alpha _{k}}-1}{%
\sum }}\frac{S_{j}f }{q_{\alpha _{k}}-j}+\frac{1}{q_{\alpha _{k}}}\underset{%
j=M_{2\alpha _{k}}}{\overset{q_{\alpha _{k}}}{\sum }}\frac{S_{j}f}{q_{\alpha
_{k}}-j}  \notag \\
&=&\frac{1}{l_{q_{\alpha _{k}}}}\underset{j=1}{\overset{M_{2\alpha _{k-1}}+1}{%
\sum }}\frac{S_{j}f }{q_{\alpha _{k}}-j}+\frac{1}{q_{\alpha _{k}}}\underset{%
j=M_{2\alpha _{k}}}{\overset{q_{\alpha _{k}}}{\sum }}\frac{S_{j}f}{q_{\alpha
_{k}}-j}  \notag \\
&:=& I+II.  \notag
\end{eqnarray}
Let $j<M_{2\alpha _{k-1}+1}.$ By combining (\ref{4}) and (\ref{8}) we get that 
\begin{eqnarray}  \label{9}
\left| S_{j}f \right| &\leq &\sum_{\eta =0}^{k-1}\sum_{v=M_{2\alpha _{\eta
}}}^{M_{2\alpha _{\eta }+1}-1}\left| \widehat{f}(v)\right|  \notag \\
&\leq &\sum_{\eta =0}^{k-1}\sum_{v=M_{2\alpha _{\eta }}}^{M_{2\alpha _{\eta
}+1}-1}\frac{1}{\sqrt{\alpha _{\eta }}}\leq c\sum_{\eta =0}^{k-1}\frac{M_{2\alpha _{\eta }}}{\sqrt{\alpha _{\eta }}} \notag \\
&\leq&c\sum_{\eta =0}^{k-2}\frac{M_{2\alpha _{\eta }}}{\sqrt{\alpha _{\eta }}}+\frac{cM_{2\alpha _{k-1}}}{
\sqrt{\alpha _{k-1}}} \notag \\
&\leq&\frac{2cM_{2\alpha _{k-1}}}{
\sqrt{\alpha _{k-1}}}. \notag
\end{eqnarray}
Since 
\begin{eqnarray}  \label{10}
\sum_{j=0}^{M_{2\alpha _{k-1}}}\frac{1}{q_{\alpha _{k}}-j}&\leq&
\sum_{j=M_{2\alpha _{k}-1}}^{M_{2\alpha_{k}}}\frac{1}{j}\leq \log
M_{2\alpha_{k}}-\log M_{2\alpha_{k}-1} \\
&\leq & \log \frac{M_{2\alpha_{k}}}{M_{2\alpha_{k}-1}}<c<\infty  \notag
\end{eqnarray}
according to (\ref{55}) and (\ref{10}) we can conclude that 
\begin{eqnarray*}
&&\left|I\right|\leq\frac{2c}{\alpha_{k}}\sum_{j=0}^{M_{2\alpha _{k-1}}}\frac{1}{q_{\alpha _{k}}-j}\frac{M_{2\alpha _{k-1}}} {\sqrt{\alpha_{k-1}}} \\ \notag
&\leq&\frac{c}{\alpha_k}
\frac{M_{2\alpha_{k-1}}}{\sqrt{\alpha_{k-1}}}<c<\infty.
\end{eqnarray*}

Hence, 
\begin{equation*}
\left\Vert I\right\Vert_1<c<\infty.  \label{17}
\end{equation*}
Let $M_{2\alpha _{k}}\leq j\leq q_{\alpha _{k}}.$ Then we have that 
\begin{eqnarray*}
S_{j}f &=&\sum_{\eta =0}^{k-1}\sum_{v=M_{2\alpha _{\eta }}}^{M_{2\alpha
_{\eta }+1}-1}\widehat{f}(v)\psi _{v} +\sum_{v=M_{2\alpha _{k}}}^{j-1}%
\widehat{f}(v)\psi _{v}  \label{11} \\
&=&\sum_{\eta =0}^{k-1}\frac{M_{2\alpha _{\eta }}^{1/p-1}}{\sqrt{\alpha
_{\eta }}}\left( D_{M_{_{2\alpha _{\eta }+1}}} -D_{M_{_{2\alpha _{\eta }}}}
\right)  \notag \\
&&+\frac{M_{2\alpha _{k}}^{1/p-1}}{\sqrt{\alpha _{k}}}\left( D_{j}
-D_{M_{_{2\alpha _{k}}}} \right) .  \notag
\end{eqnarray*}
This gives that 
\begin{eqnarray*}
II &=&\frac{1}{l_{q_{\alpha _{k}}}}\sum_{j=M_{2\alpha _{k}}}^{q_{\alpha
_{k}}}\frac{1}{q_{\alpha _{k}}-j}\left( \sum_{\eta =0}^{k-1}\frac{1}{\sqrt{%
\alpha _{\eta }}}\left( D_{M_{_{2\alpha _{\eta }+1}}}-D_{M_{_{2\alpha _{\eta
}}}} \right) \right)  \label{18} \\
&&+\frac{1}{l_{q_{\alpha _{k}}}}\frac{1}{\sqrt{\alpha _{k}}}%
\sum_{j=M_{2\alpha _{k}}}^{q_{\alpha _{k}}}\frac{\left(
D_{j}-D_{M_{_{2\alpha _{k}}}} \right) }{q_{\alpha _{k}}-j}:=II_{1}+II_{2}. 
\notag
\end{eqnarray*}
Applying (\ref{4}) in $II_{1}$ we have that 
\begin{eqnarray*}  \label{19}
\left\Vert II_{1}\right\Vert_1 &\leq &\frac{1}{l_{q_{\alpha _{k}}}}%
\sum_{j=M_{2\alpha _{k}}}^{q_{\alpha _{k}}}\frac{1}{q_{\alpha _{k}}-j}
\sum_{\eta =0}^{k-1}\frac{1}{\sqrt{\alpha _{\eta }}}\left\Vert
D_{M_{_{2\alpha _{\eta }+1}}}-D_{M_{_{2\alpha _{\eta }}}} \right\Vert_1 \\
&\leq & \frac{1}{l_{q_{\alpha _{k}}}}\sum_{j=M_{2\alpha _{k}}}^{q_{\alpha
_{k}}}\frac{1}{q_{\alpha _{k}}-j} \sum_{\eta =0}^{\infty}\frac{1}{\sqrt{%
\alpha _{\eta }}}<c<\infty.  \notag
\end{eqnarray*}

Since 
\begin{equation}  \label{141}
D_{j+M_{n}}=D_{M_{n}} +\psi_{M_n}D_{j} ,\text{ \thinspace when \thinspace }%
j<M_{n},
\end{equation}
for $II_{2}$ we obtain that 
\begin{eqnarray*}
II_{2}&=&\frac{1}{l_{q_{\alpha _{k}}}}\frac{1}{\sqrt{\alpha _{k}}}%
\psi_{M_{2\alpha _{k}}}\sum_{j=0}^{q_{\alpha _{k}-1}-1}\frac{ D_{j}}{%
q_{\alpha _{k}-1}-j} \\
&=&\frac{l_{q_{\alpha_{k}}-1}}{l_{q_{\alpha _{k}}}}\frac{1}{\sqrt{\alpha _{k}%
}}\psi_{M_{2\alpha _{k}}}F_{q_{\alpha _{k}}-1}.  \label{20}
\end{eqnarray*}

Hence, if we apply Lemma \ref{bt} in the case when $ \delta=1/8 $ and $ \beta=3/4 $ for the sufficiently large $k $ we can estimate
as follows 
\begin{eqnarray*}
\left\Vert L_{q_{\alpha _{k}}}f \right\Vert_1 &\geq & \left\Vert
II_2\right\Vert_1-\left\Vert II_1\right\Vert_1-\left\Vert I\right\Vert_1 \\
&\geq & \frac{1}{2}\left\Vert II_2\right\Vert_1 \geq \frac{c\left\Vert
F_{q_{\alpha _{k}}-1} \right\Vert_1}{\sqrt{\alpha _{k}}} \\
&\geq &\frac{c\alpha_{k}^{3/4}}{\sqrt{\alpha_{k}}}\geq c\alpha_{k}^{1/4}\rightarrow \infty ,\text{ as }k\rightarrow \infty.
\end{eqnarray*}

The proof is complete.
\end{proof}

\begin{proof}[Proof of Theorem 2.]
It is obvious that 
\begin{equation*}
\underset{n\in \mathbb{N}}{\sup }\frac{\left\vert L_{n}f\right\vert }{\log
\left( n+1\right) }\leq \frac{1}{l_{n}}\sum_{k=0}^{n}\frac{1}{(n-k)}\underset%
{k\in \mathbb{N}}{\sup }\frac{\left\vert S_{k}f\right\vert }{\log \left(
k+1\right) }\leq \sup_{n\in \mathbb{N}}\frac{\left\vert S_{n}f\right\vert }{%
\log \left( n+1\right) }.
\end{equation*}
By using Theorem T1 we can conclude that%
\begin{equation*}
\left\Vert \underset{n\in \mathbb{N}}{\sup }\frac{\left\vert
L_{n}f\right\vert }{\log \left( n+1\right) }\right\Vert _{1}\leq \left\Vert
\sup_{n\in \mathbb{N}}\frac{\left\vert S_{n}f\right\vert }{\log \left(
n+1\right) }\right\Vert _{1}\leq c\left\Vert f\right\Vert _{H_{1}}.
\end{equation*}
The proof of part a) is complete.

Under condition (\ref{5}), there exists a positive integers $\ \left\{
n_{k};k\in \mathbb{N}_{+}\right\} \subset \left\{ \lambda _{k};k\in \mathbb{N%
}_{+}\right\} $ such that 
\begin{equation*}
\lim_{k\rightarrow \infty }\frac{n_{k}}{\varphi \left( q_{n_{k}}\right) }%
=\infty .
\end{equation*}
Set 
\begin{equation*}
f_{n_{k}}=D_{M_{2n_{k}+1}}-D_{M_{_{2n_{k}}}}.
\end{equation*}
It is evident 
\begin{equation*}
\widehat{f}_{n_{k}}\left( i\right) =\left\{ 
\begin{array}{l}
\text{ }1,\text{ if }i=M_{2n_{k}},...,M_{2n_{k}+1}-1, \\ 
\text{ }0,\text{otherwise.}%
\end{array}%
\right.
\end{equation*}
Then we have that 
\begin{equation}
S_{i}f_{n_{k}}=\left\{ 
\begin{array}{l}
D_{i}-D_{M_{_{2n_{k}}}},\text{ }\text{ if }
i=M_{_{2n_{k}}}+1,...,M_{2n_{k}+1}-1, \\ 
\text{ }f_{n_{k}},\text{ \ \ }\text{ \ \ }\text{ \ \ }\text{ \ \ }\text{ if }%
i\geq M_{2n_{k}+1}, \\ 
0,\text{ \ \ }\text{ \ \ }\text{ \ \ }\text{ \qquad otherwise.}%
\end{array}%
\right.  \label{14}
\end{equation}
From (\ref{normsup}) and (\ref{14}) we get that 
\begin{equation*}
\left\Vert f_{n_{k}}\right\Vert _{H_{1}}<c<\infty .
\end{equation*}
It is easy to show that 
\begin{eqnarray*}
L_{q_{n_{k}}}f_{n_{k}} &=&\frac{1}{l_{q_{n_{k}}}}\underset{j=1}{\overset{%
q_{n_{k}}}{\sum }}\frac{S_{j}f_{n_{k}}}{q_{n_{k}}-j} \\
&=&\frac{1}{l_{q_{n_{k}}}}\underset{j=M_{_{2n_{k}}}+1}{\overset{q_{n_{k}}}{%
\sum }}\frac{D_{j}-D_{M_{_{2n_{k}}}}}{q_{n_{k}}-j} \\
&=&\frac{1}{l_{q_{n_{k}}}}\underset{j=1}{\overset{q_{n_{k}-1}-1}{\sum }}%
\frac{D_{j+M_{_{2n_{k}}}}-D_{M_{_{2n_{k}}}}}{q_{n_{k}-1}-j}.
\end{eqnarray*}

By applying (\ref{141}) we find that 
\begin{equation*}
\left\vert L_{q_{n_{k}}}f_{n_{k}}\right\vert =\frac{1}{l_{q_{n_{k}}}}%
\left\vert \underset{j=1}{\overset{q_{n_{k}-1}}{\sum }}\frac{D_{j}}{%
q_{n_{k}-1}-j}\right\vert =\frac{l_{q_{n_{k}-1}}}{l_{q_{n_{k}}}}\left\vert
F_{q_{n_{k}-1}}\right\vert .
\end{equation*}

By now using Lemma \ref{ptw1} we can conclude that 
\begin{equation*}
\int_{G_{m}}\frac{\left\vert L_{q_{n_{k}}}f_{n_{k}}\right\vert }{\varphi
\left( {q_{n_{k}}}\right) }d\mu \geq \frac{c\left\Vert
F_{q_{n_{k}-1}}\right\Vert }{\varphi \left( q_{n_{k}}\right) }\geq \frac{%
cn_{k}}{\varphi \left( q_{n_{k}}\right) }\rightarrow \infty ,\text{ as }%
k\rightarrow \infty .
\end{equation*}

The proof is complete.

\textbf{Acknowledgment:} The author would like to thank the referee for helpful suggestions, which absolutely improve the final version of our paper.
\end{proof}

\end{document}